\documentclass[final]{siamltex}

\usepackage{amssymb,amsmath,amscd}
\usepackage{algorithmic,algorithm}
\usepackage{dsfont}
\usepackage{graphicx}
\usepackage{color}
\usepackage{subfigure}
\usepackage[small,bf]{caption}
\usepackage{framed}
\usepackage{todonotes}
\usepackage{enumerate}
\usepackage{wrapfig}
\usepackage{dpfloat}

\newcommand{\n}[1]{\mathbf{ #1}}

\newcommand{\nv}{\n{n}}

\newcommand{\tv}{{\boldsymbol{\tau}}}

\newcommand{\Grad}{\nabla}

\newcommand{\sign}{\mathrm{sgn}}
\newcommand{\R}{\mathbb{R}}
\newcommand{\C}{\mathbb{C}}

\newcommand{\vectornorm}[1]{\left \|#1 \right \|}

\newcommand{\abs}[1]{\left | #1 \right |}

\newcommand{\admshape}{\mathcal{D}}

\newcommand{\alphaspace}{\mathcal{A}}

\newcommand{\dint}{\,d}

\newcommand{\gt}{\tilde{g}}
\newcommand{\Omegab}{\bar{\Omega}}

\newcommand{\Psit}{\tilde{\Psi}}
\newcommand{\omegat}{\tilde{\omega}}

\DeclareMathOperator{\minimize}{minimize}

\renewcommand\Re{\operatorname{Re}}

\newtheorem{remark}[theorem]{Remark}

\title{A Numerical Approach to Shape Optimization with State Constraints}

\author{C. Leith\"auser
	\thanks{Department of Mathematics, TU Kaiserslautern, Kaiserslautern, Germany. 
	Email: christian.leithaeuser@itwm.fraunhofer.de}
        \and R. Pinnau
        \thanks{Department of Mathematics, TU Kaiserslautern, Kaiserslautern, Germany.}
        \and R. Fe{\ss}ler
        \thanks{Fraunhofer ITWM, Kaiserslautern, Germany.}}

\begin{document}

\maketitle

\begin{abstract}
 We present a general numerical approach to shape optimization with state constraints for 2-dimensional geometries, without relaxing the constraints. To do this we reformulate the problem on a fixed reference domain using conformal pull-back. The shape dependence is then hidden in a conformal parameter, which appears as a coefficient in the differential operators. The problem on the reference domain can be discretized, leading to an NLP which can be handled using existing solvers. Furthermore, we deal with the question how constraints on the conformal parameter can be used to preserve characteristic features of the geometry. We introduce this approach with the help of a stokes flow, where the task is finding a shape such that the wall shear stress is supremum norm close to some given target.
\end{abstract}



\pagestyle{myheadings}
\thispagestyle{plain}

\section{Introduction}

In this paper we present a general numerical approach to shape optimization with state constraints. Our line of action is to first reformulate the optimization problem on a fixed reference domain, using conformal pull-back. This leads to a nonlinear elliptic optimal control problem with state constraints which can be discretized and solved by nonlinear programming (NLP) techniques.

The existing strategies for shape optimization problems with state constraints include treating the constraint through a penalty term in the cost functional (cf. \cite{ramakrishnan1974structural,mohammadi2001applied}). This however does not assure that the constraint is fulfilled in a strict sense. Optimal control problems with state constraints, i.e. without shape dependence, are actively studied in the literature. For the general theory and several applications we refer to \cite{hinze2009optimization}. First order necessary and second order sufficient conditions are derived in \cite{casas1993boundary,casas2000second}. Numerical approaches to a variety of problems can be found in \cite{bergounioux1997augmented,maurer2000optimization, maurer2001optimization,grund2001optimal,hintermüller2009moreau}. A common praxis, which assures that the state constraints are treated in a strict sense, is to discretize the control problem leading to a nonlinear programming problem which can be solved using NLP-techniques. There are basically two options for the discretization concept: One can either discretize both control and state variables and implement their relation explicitly through equality constraint. Or one can treat the discretized control as the only optimization variable and compute the state as a function of the control. See \cite{grund2001optimal} for a comparison of these approaches. For our case we utilize the first setting, such that both control and state appear as variables of the NLP. Especially for nonlinear problems the choice of the NLP-solver is of great importance. As suggested in \cite{maurer2000optimization} we use the interior point method developed in \cite{vanderbei1999interior}. 

For an overview of the general theory of shape optimization we refer to \cite{pironneau1984optimal,sokolowski1992introduction,mohammadi2001applied,harbrecht2008analytical}. See \cite{eppler2008convergence} for existence and convergence results of general elliptic shape optimization problems. See \cite{allaire2004structural, penzler2010phase} for examples of structural optimization. However, most concepts from standard theory do not apply to problems with state constraints. At least not without relaxing the constraints, which we want to avoid. For our approach we exploit the Riemann mapping theorem (cf. \cite{schinzinger2003conformal}) which states that any two simply connected domains in $\R^2$ can be mapped onto each other by conformal maps. Conformal maps however are determined by a scalar function which we call the conformal parameter. We can utilize this by pulling-back the optimization system to a fixed reference domain, where the shape dependence is then hidden in the conformal parameter which influences the differential operator as a coefficient. Thus, we can optimize on the reference domain and reconstruct the optimal domain later from the optimal conformal parameter.

We introduce the approach with the help of a Stokes flow with supremum norm cost functional. The task is to find a domain $\Omega$ such that the wall shear stress on the boundary is supremum norm close to some given target wall shear stress. Our application in view is the improvement of polymer distributors as they are used in fiber production. Another application from hemodynamics can be found in \cite{quarteroni2003optimal} and \cite{rozza2005optimization}.

We begin by introducing basic concepts about conformal maps in Section \ref{sec:conformal maps}. The flow problem under consideration is given in Section \ref{sec:stokes flow} and a shape optimization problem with supremum norm cost functional is formulated in Section \ref{sec:linf:optimization problem}. The existence of an optimal control is shown in Section \ref{sec:ex optimal control}. Since the influence of the conformal parameter is global, Section \ref{sec:inflow constraints} deals with the question of how characteristic features of the geometry, like the shape of the inflow boundaries, can be preserved by applying constraints to the conformal parameter. Then, Section \ref{state:sec:discretization} explains how the NLP is obtained through discretization by finite elements. Numerical results are presented in Section \ref{sec:numerical results} after which we close with a conclusion.

\section{Conformal Maps}
\label{sec:conformal maps}
Conformal maps are a special class of diffeomorphisms which are angle preserving. We use them to pull-back a shape-dependent problem to a fixed reference domain. The shape information is then hidden in a so called conformal parameter, which is a scalar function living on the reference domain. 
\begin{definition}
 Let $\Omega_0, \Omega_\alpha \subset \R^2$ be two-dimensional domains. Then a $k$-diffeomorphism $T = (T_1, T_2): \Omega_0 \rightarrow \Omega_\alpha$, $k \geq 1$ is called conformal map, if it fulfills the Cauchy-Riemann equations
 \begin{align}
 \begin{aligned}
  \partial_1 T_1 &= \partial_2 T_2
  \\
  \partial_2 T_1 &= -\partial_1 T_2
 \end{aligned}
 \end{align}
 on $\Omega_0$. Therefore, it is possible to identify conformal maps with holomorphic complex functions. We define the conformal parameter $\alpha \in C^{k-1}(\Omegab_0)$ such that
 \begin{align}
  e^{2\alpha} = \det(D T).
 \end{align}
In the following we write $T_\alpha: \Omega_0 \rightarrow \Omega_\alpha$ for a conformal map corresponding to the conformal parameter $\alpha$.
\end{definition}

Conformal maps can also be defined in higher dimensions, however, already in three dimensions the set of reachable domains is negligible small. On the other hand in two dimensions the Riemann Mapping Theorem states that all simply connected domains can be reached from a simply connected reference domain by conformal deformations:
\begin{theorem}[Riemann Mapping Theorem, see \cite{schinzinger2003conformal}]
 Let $\Omega_0, \Omega_1 \subset \R^2$ be two sufficiently regular simply connected domains. Then, there exists a conformal map $T: \Omega_0 \rightarrow \Omega_1$.
\end{theorem}

The Riemann Mapping Theorem signifies why it makes sense to use conformal shape deformations for two-dimensional shape problems. It shows that the conformal approach does not restrict the set of reachable shapes. The advantage of using this approach is that it enables us to reformulate problems on the reference domain with coefficients depending on the conformal parameter. To do this we proof the following result which shows for what conformal parameters a corresponding conformal map exists. The proof uses arguments for holomorphic functions (cf. \cite{fischer1980funktionentheorie}), therefore, we identify the conformal map with a complex function.
\begin{lemma}
\label{lemma:existence conformal map}
 Let $\Omega_0 \subset \R^2$ be simply connected and assume that $\alpha \in C^k(\Omegab_0)$, $k \geq 2$ is harmonic, i.e. $\Delta \alpha = 0$. Then, there exists a conformal map $T_\alpha: \Omega_0 \rightarrow \Omega_\alpha$ with conformal parameter $\alpha$. Furthermore, $T_\alpha$ is unique up to global translation and rotation of $\Omega_\alpha$.
\end{lemma}
\begin{proof}
 We identify $\R^2$ with the complex plane $\C$. Because $\alpha$ is harmonic we know from \cite{henrici1993applied} that there exists a holomorphic function $g: \Omega_0 \rightarrow \C$ such that $\Re g = \alpha$. We write $g = \alpha + \imath \beta$ for some imaginary part $\beta$. We know that $\beta$ is uniquely determined up to a constant which we can fix by $\beta(z_0) = \beta_0 \in \R$ for a certain point $z_0 \in \Omega_0$. Then, $e^g$ is also holomorphic and thus there exists a unique holomorphic function $T_\alpha \in C^{k+1}(\Omega_0, \C)$ such that
 \begin{align}
 \begin{aligned}
  \partial_z T_\alpha &= e^g
  \\
  T_\alpha(z_0) &= y_0
 \end{aligned}
 \end{align}
for some $y_0 \in \C$. The value $T_\alpha(z_1)$ can be obtainted by integration over an arbitrary path from $z_0$ to $z_1 \in \Omega_0$. And since $\Omega_0$ is simply connected it is independent to the choice of the path (cf. \cite{fischer1980funktionentheorie}). On the one hand
\begin{align}
\begin{aligned}
 \partial_z T_\alpha \overline{\partial_z T_\alpha} &= e^g \overline{e^g}
 = \abs{e^g}^2 = \abs{e^{\alpha + \imath \beta}}^2 
 \\
 &= \abs{e^\alpha e^{\imath \beta}}^2
 = e^{2 \alpha}
\end{aligned}
\end{align}
and on the other hand when identifying $T_\alpha$ with the corresponding diffeomorphism we get
\begin{align}
 \begin{aligned}
  \partial_z T_\alpha \overline{\partial_z T_\alpha}
  = \det(D T_\alpha)
 \end{aligned}
\end{align}
which yields $\det(D T_\alpha) = e^{2\alpha}$. Defining $\Omega_\alpha = T_\alpha(\Omega_0)$, we have constructed a conformal diffeomorphism $T_\alpha: \Omega_0 \rightarrow \Omega_\alpha$ for the conformal parameter $\alpha$.

By construction the conformal map is unique up to the choice of constants $y_0 \in \C$ and $\beta_0 \in \R$. Then, $y_0$ can be used for global translations and $\beta_0$ induces a rotation of $\Omega_\alpha$ around the point $y_0$.
\end{proof}

We have seen that the conformal map for a given conformal parameter is unique up to global translation and rotation of the conformal domain $\Omega_\alpha$. This means that the shape is unique and just the embedding into the $\R^2$ plane is undetermined. We take equivalence classes to get a unique relation between conformal parameter and corresponding conformal domain. The problem under consideration is well-defined with respect to these equivalence classes.

\begin{remark}
Note that it may happen that for certain conformal parameters the conformal map constructed in Lemma \ref{lemma:existence conformal map} produces a self-overlapping domain which cannot be embedded into the two-dimensional plane. In this case $T_\alpha$ would not be bijective. But we could restore the bijectivity by interpreting $\Omega_\alpha$ as a suitable defined manifold.
\end{remark}

\section{Stokes Flow}
\label{sec:stokes flow}
Let $\Omega_0 \subset \R^2$ be a simply connected bounded reference domain of class $C^{4,1}$. Let the boundary $\Gamma_0$ decompose into the in- and outflow parts $\Gamma_0^{in}$ and the wall parts $\Gamma_0^w$. We consider an optimization problem with supremum norm cost functional based on the Stokes flow together with conformal shape variations. Let the admissible set of conformal parameters 
\begin{align}
\alphaspace \subset \{ \alpha \in H^4(\Omega_0); \Delta \alpha = 0 \}
\end{align}
be given. 
\begin{remark}
Further constraints on $\alphaspace$ are reasonable to preserve certain features of the geometry, like the shape of the inflow boundaries. In Section \ref{sec:inflow constraints} we discuss this subject further and introduce specific choices for $\alphaspace$, which we examine with the help of numerical examples in Section \ref{sec:numerical results}.
\end{remark}

Since $H^4(\Omega_0)$ embeds into $C^2(\Omegab_0)$, Lemma \ref{lemma:existence conformal map} yields the existence of the conformal map $T_\alpha: \Omega_0 \rightarrow \Omega_\alpha$ corresponding to $\alpha \in \alphaspace$. Therefore, we can define the set of admissible shapes by
\begin{align}
 \admshape = \{ \Omega_\alpha = T_\alpha(\Omega_0); \alpha \in \alphaspace \}.
\end{align}
Furthermore, we define
\begin{align}
\begin{aligned}
 \Gamma_\alpha^{in} &= T_\alpha(\Gamma_0^{in})
 \\
 \Gamma_\alpha^{w} &= T_\alpha(\Gamma_0^{w})
\end{aligned}
\end{align}
for $\Omega_\alpha \in \admshape$.

Let $g_0 \in H^\frac{7}{2}(\Gamma_0)$ be a given reference inflow condition with 
\begin{align}
\label{eq:d_tv g_0 = 0}
\partial_s g_0 = 0 \qquad \mbox{on $\Gamma_0^w$},
\end{align}
 where $\partial_s$ denotes the tangential derivative at the boundary. This means, as we see in \eqref{eq:no flow through wall}, that there is no flow through the wall boundaries. For every $\Omega_\alpha \in \admshape$ we define 
\begin{align}
 g_\alpha := g_0 \circ B_\alpha^{-1}: \Gamma_\alpha \rightarrow \R
\end{align}
where $B_\alpha: \Gamma_0 \rightarrow \Gamma_\alpha$ is a diffeomorphism to be defined later. Let $(\Psi(\alpha), \omega(\alpha)) \in H^4(\Omega_\alpha) \times H^2(\Omega_\alpha)$ be the unique solution of
\begin{align}
\label{eq:state bhstokes 1}
\begin{aligned}
 \Delta \Psi(\alpha) &= -\omega(\alpha) \qquad &&\mbox{in $\Omega_\alpha$}
 \\
 \Delta \omega(\alpha) &= 0 \qquad &&\mbox{in $\Omega_\alpha$}
 \\
 \Psi(\alpha) &= g_\alpha = g_0 \circ B_\alpha^{-1} \qquad &&\mbox{on $\Gamma_\alpha$}
 \\
 \partial_\nv \Psi(\alpha) &= 0 \qquad &&\mbox{on $\Gamma_\alpha$}
 \end{aligned}
\end{align}
where $\Psi(\alpha)$ and $\omega(\alpha)$ are called stream function and vorticity. By defining the flow velocity
\begin{align}
 \n{u}(\alpha) = 
 \begin{pmatrix} \partial_2 \Psi(\alpha) \\ -\partial_1 \Psi(\alpha) \end{pmatrix}
\end{align}
this biharmonic problem is equivalent to the incompressible Stokes problem (cf. \cite{anderson1995computational}). The boundary conditions lead to the following velocity conditions: On the boundary
\begin{align}
\begin{aligned}
\label{eq:no flow through wall}
 \nv \cdot \n{u}(\alpha) &= \partial_s \Psi(\alpha) = \partial_s g_\alpha
 \qquad \mbox{on $\Gamma_\alpha$}
\end{aligned}
\end{align}
which vanishes on $\Gamma_\alpha^w$ due to \eqref{eq:d_tv g_0 = 0} and
\begin{align}
 \tv \cdot \n{u} = -\partial_\nv \Psi = 0
 \qquad \mbox{on $\Gamma_\alpha$}.
\end{align}
Furthermore the wall shear stress which we denote by $\sigma$ is equal to the vorticity evaluated on the wall boundaries
\begin{align}
 \sigma(\alpha) = \omega(\alpha)|_{\Gamma_\alpha^w}.
\end{align}

The following lemma states the existence and regularity of the solution of \eqref{eq:state bhstokes 1} and provides the equivalent pull-back formulation on the reference domain.
\begin{lemma}
\label{lemma:linf:regularity}
 Let $\alpha \in \alphaspace \subset \{ \alpha \in H^4(\Omega_0); \Delta \alpha = 0 \}$ and define $\Psit(\alpha) := \Psi(\alpha) \circ T_\alpha$ and $\omegat(\alpha) := \omega(\alpha) \circ T_\alpha$. Then \eqref{eq:state bhstokes 1} is equivalent to
 \begin{align}
\label{eq:state bhstokes 2}
\begin{aligned}
 \Delta \Psit(\alpha) &= -e^{2\alpha} \omegat(\alpha) \qquad &&\mbox{in $\Omega_0$}
 \\
 \Delta \omegat(\alpha) &= 0 \qquad &&\mbox{in $\Omega_0$}
 \\
 \Psit(\alpha) &= g_0 \circ B_\alpha^{-1} \circ T_\alpha \qquad &&\mbox{on $\Gamma_0$}
 \\
 \partial_\nv \Psit(\alpha) &= 0 \qquad &&\mbox{on $\Gamma_0$}.
 \end{aligned}
\end{align}
Assume that $B_\alpha = T_\alpha$ or $B_\alpha$ is sufficiently regular. Then the solution is unique and $(\Psit(\alpha), \omegat(\alpha)) \in H^4(\Omega_0) \times H^2(\Omega_0)$ and especially $\omegat(\alpha)|_{\Gamma_0^w} \in C^0(\Gamma_0^w)$.
\end{lemma}
\begin{proof}
 Let $\alpha \in \alphaspace$. The equivalence of \eqref{eq:state bhstokes 1} and \eqref{eq:state bhstokes 2} follows from \cite{schinzinger2003conformal}. Define the bilinear form
 \begin{align}
  a(\psi, \phi) := \int_{\Omega_0} e^{-2 \alpha} \Delta \psi \, \Delta \phi \dint x.
 \end{align}
We have $\alpha \in C^2(\Omegab_0)$, therefore, $e^{-2\alpha}$ is positive and bounded away from zero and infinity. Then, standard existence and regularity theory (cf. \cite{wloka1987partial}) yields the existence of a unique solution $(\Psit(\alpha), \omegat(\alpha)) \in H^4(\Omega_0) \times H^2(\Omega_0)$. Finally, $\omegat(\alpha)|_{\Gamma_0^w} \in C^0(\Gamma_0^w)$ follows from the Lemma of Sobolev (see \cite{wloka1987partial}).
\end{proof}

\subsection{Push-Forward of the Inflow Condition}
It remains to define the map $B_\alpha: \Gamma_0 \rightarrow \Gamma_\alpha$ which is used to push-forward the inflow condition. We consider two choices: The simplest way is to use the conformal map itself, i.e.
\begin{align}
\label{eq:state:ConformalMapping}
\tag{$\mathcal{B}1$}
 B_\alpha := T_\alpha.
\end{align}
Unfortunately, this can alter the velocity inflow condition, i.e. $\nv \cdot \n{u}(0) = B_\alpha^*(\nv \cdot \n{u}(\alpha))$ would not hold. To see this remember that the velocity corresponding to the stream function $\Psi(\alpha)$ is
\begin{align}
 \n{u}(\alpha) = 
 \begin{pmatrix} \partial_2 \Psi(\alpha) \\ -\partial_1 \Psi(\alpha) \end{pmatrix}.
\end{align}
Then,
\begin{align}
\label{eq:velocity mapping}
\begin{aligned}
 \nv \cdot \n{u}(0) &= \partial_s \Psi(0) = \partial_s g_0
 = \partial_s (\Psi(\alpha) \circ B_\alpha)
 = B_\alpha^*(\partial_s \Psi(\alpha)) \, \partial_s B_\alpha
 \\
 &= B_\alpha^*(\nv \cdot \n{u}(\alpha)) \, \partial_s B_\alpha
\end{aligned}
\end{align}
and for \eqref{eq:state:ConformalMapping}, $\partial_s B_\alpha = \partial_s T_\alpha \neq 1$ in general.

To fulfill $\nv \cdot \n{u}(0) = B_\alpha^*(\nv \cdot \n{u}(\alpha))$ we define $I_\alpha: \Gamma_0 \rightarrow \Gamma_\alpha$ such that
\begin{align}
\begin{aligned}
 &I_\alpha|_{\Gamma_0^{in}}: \Gamma_0^{in} \rightarrow \Gamma_\alpha^{in}  \quad \mbox{is isometric}
 \\
 &I_\alpha|_{\Gamma_0^w} = T_\alpha|_{\Gamma_0^w}.
\end{aligned}
\end{align}
Of course such a map can only exist if the corresponding inflow parts have the same length, i.e. if for every $\Omega_\alpha \in \admshape$
\begin{align}
\label{state:eq:isometry assumption}
\int_{\Gamma_0^k} \dint s = \int_{T_\alpha(\Gamma_0^k)} \dint s
\qquad \mbox{for every inflow part $\Gamma_0^k \in \mathcal{C}(\Gamma_0^{in})$},
\end{align}
where the set of connected components of the inflow boundary is defined by
\begin{align}
 \mathcal{C}(\Gamma_0^{in}) := \{ \Gamma_0^k \subset \Gamma_0^{in};\; \mbox{$\Gamma_0^k$ is connected component of $\Gamma_0^{in}$} \}.
\end{align}
Then we can use
\begin{align}
 \label{eq:state:IsometricMapping}
 \tag{$\mathcal{B}2$}
 B_\alpha := I_\alpha
\end{align}
to push-forward the boundary condition.

\begin{lemma}
 If $B_\alpha = I_\alpha$, then for every $\Omega_\alpha \in \admshape$ which fulfills \eqref{state:eq:isometry assumption}, $\nv \cdot \n{u}(0) = B_\alpha^*(\nv \cdot \n{u}(\alpha))$ holds and therewith we get a realistic mapping of the inflow condition.
\end{lemma}
\begin{proof}
We have
\begin{align}
\begin{aligned}
 \nv \cdot \n{u}(0) = I_\alpha^*(\nv \cdot \n{u}(\alpha)) \, \partial_s I_\alpha
\end{aligned}
\end{align}
by \eqref{eq:velocity mapping}. Then, by definition $\partial_s I_\alpha = 1$ on $\Gamma_0^{in}$ and $\nv \cdot \n{u}(\alpha) = 0$ on $\Gamma_0^w$ which yields the result.
\end{proof}

Using \eqref{eq:state:IsometricMapping} has the advantage that the inflow condition is mapped isometrically. However, this does increase the complexity of the optimization problem, because the isometry must be computed which adds additional nonlinearities, as we are going to see in Section \ref{sec:inflow constraints}.

\section{Optimization Problem}
\label{sec:linf:optimization problem}

We can now formulate the shape optimization problem with supremum norm cost functional. Let $\sigma_d \in C^0(\Gamma_0^w)$ be a given target wall shear stress and let $\varepsilon \geq 0$. We want to solve the problem
\begin{align}
\label{eq:state:optsystem1}
\begin{aligned}
   \minimize_{(\Omega_\alpha, \Psi, \omega) \in M_1}
   &\quad \vectornorm{\sigma_d - T_\alpha^*\omega}_{C^0(\Gamma_0^w)} 
   + \varepsilon \vectornorm{\alpha}_{H^4(\Omega_0)}^2
   \\
   \mbox{with} \quad M_1 &= \admshape \times H^2(\Omega_\alpha) \times H^2(\Omega_\alpha)
   \\
   \mbox{subject to}\quad \Delta \Psi &= -\omega \qquad &&\mbox{in $\Omega_\alpha$}
   \\
   \quad \Delta \omega &= 0 \qquad &&\mbox{in $\Omega_\alpha$}
   \\
   \Psi &= g_0 \circ B_\alpha^{-1} \qquad &&\mbox{on $\Gamma_\alpha$}
   \\
   \partial_\nv \Psi &= 0 \qquad &&\mbox{on $\Gamma_\alpha$}.
\end{aligned}
\end{align}
Thus the task is to minimize the supremum norm distance between wall shear stress $\sigma = \omega|_{\Gamma_\alpha^w}$ and target wall shear stress.

\begin{remark}
 Note that the high regularity of the control space $\alphaspace \subset H^4(\Omega_0)$ and regularization term is necessary to assure that $\omega \in C^0(\Omegab)$ holds (cf. Lemma \ref{lemma:linf:regularity}) and thus that the optimal control problem is well-defined. On the other hand, if we would use an $L^2$-cost functional instead of a $C^0$-functional, we would require less regularity.
\end{remark}

\subsection{Shape Problem on Reference Domain}
  We eliminate the shape-dependence by applying the conformal pull-back operator $T_\alpha^*$ to the whole system. Using Lemma \ref{lemma:linf:regularity} yields a new optimization problem on the reference domain $\Omega_0$ which is equivalent to \eqref{eq:state:optsystem1}. Instead of $\Omega_\alpha \in \admshape$ the conformal parameter $\alpha \in \alphaspace$ acts as the control.
  \begin{align}
\label{eq:state:optsystem2}
\begin{aligned}
   \minimize_{(\alpha, \Psi, \omega) \in M_2}
   &\quad \vectornorm{\sigma_d - \omega}_{C^0(\Gamma_0^w)} 
   + \varepsilon \vectornorm{\alpha}_{H^4(\Omega_0)}^2
   \\
   \mbox{with} \quad M_1 &= \alphaspace \times H^2(\Omega_0) \times H^2(\Omega_0)
   \\
   \mbox{subject to}\quad \Delta \Psi &= -e^{2\alpha}\omega \qquad &&\mbox{in $\Omega_0$}
   \\
   \quad \Delta \omega &= 0 \qquad &&\mbox{in $\Omega_0$}
   \\
   \Psi &= g_0 \circ B_\alpha^{-1} \circ T_\alpha \qquad &&\mbox{on $\Gamma_0$}
   \\
   \partial_\nv \Psi &= 0 \qquad &&\mbox{on $\Gamma_0$}.
\end{aligned}
\end{align}

\subsection{Shape Problem with State Constraints}
In order to eliminate the supremum norm from the cost functional, we use a standard technique (e.g.. \cite{grund2001optimal}) and replace it by a scalar variable $\delta \in \R$ together with additional inequality constraints which make sure that the distance between $\omega$ and $\sigma_d$ does not grow bigger than $\delta$. This yields
  \begin{subequations}
  \label{eq:state:OptSystemConformal}
  \begin{align}
  \label{eq:state:OptSystemConformal a}
   \mbox{minimize}_{(\delta, \alpha, \Psi, \omega) \in M_3}&\quad \delta
   +\varepsilon \vectornorm{\alpha}^2_{H^4(\Omega_0)}
   \\
   \label{eq:state:OptSystemConformal b}
   \mbox{with} \quad M_3 &= \R \times \alphaspace \times H^2(\Omega_0) \times H^2(\Omega_0)
   \\
   \label{eq:state:OptSystemConformal c}
   \mbox{subject to}\quad \Delta \Psi &= -e^{2\alpha}\omega \qquad &&\mbox{in $\Omega_0$}
   \\
   \label{eq:state:OptSystemConformal d}
   \quad \Delta \omega &= 0 \qquad &&\mbox{in $\Omega_0$}
   \\
   \label{eq:state:OptSystemConformal e}
   \Psi &= g_0 \circ B_\alpha^{-1} \circ T_\alpha \qquad &&\mbox{on $\Gamma_0$}
   \\
   \label{eq:state:OptSystemConformal f}
   \partial_\nv \Psi &= 0 \qquad &&\mbox{on $\Gamma_0$}
   \\
   \label{eq:state:OptSystemConformal g}
   \sigma_d - \omega &\leq \delta \qquad &&\mbox{on $\Gamma_0^w$}
   \\
   \label{eq:state:OptSystemConformal h}
   -\sigma_d + \omega &\leq \delta \qquad &&\mbox{on $\Gamma_0^w$}.
  \end{align}
  \end{subequations}
  This is a nonlinear optimal control problem given on a fixed domain and we can use established methods to compute the solution. All geometric information is hidden in the conformal parameter $\alpha \in \alphaspace$ and the optimal shape can be recovered later after the optimal $\alpha$ has been computed (see Section \ref{state:sec:reconstruction}). Note that in the case \eqref{eq:state:ConformalMapping} the Dirichlet condition simplifies to $\Psi = g_0$ on $\Gamma_0$.

\section{Existence of an Optimal Control}
\label{sec:ex optimal control}
To analyze the minimization problem and to form a basis for its numerical treatment we show the existence of an optimal control. The proof is straightforward and the idea can be found in \cite{hinze2009optimization}. For simplicity we only consider the conformal case \eqref{eq:state:ConformalMapping}, since the isometric case \eqref{eq:state:IsometricMapping} would involve some regularity issues at the intersection of the inflow and wall boundaries. We use
\begin{align}
\label{eq:ex min alphaspace}
 \alphaspace := \{ \alpha \in H^4(\Omega_0); \Delta \alpha = 0 \}
\end{align}
as the control space. Then Lemma \ref{lemma:linf:regularity} yields the existence of a unique state $(\Psi(\alpha), \omega(\alpha)) \in H^4(\Omega_0) \times H^2(\Omega_0)$ and Lemma \ref{lemma:existence conformal map} assures the existence of a corresponding conformal domain. Later in Section \ref{sec:inflow constraints}, where we start to deal with the application, we redefine $\alphaspace$ and add additional constraints in order to preserve certain features of the inflow boundaries. However, to keep the existence proof simple we use \eqref{eq:ex min alphaspace} for now.

\begin{theorem}
 Let $\alphaspace := \{ \alpha \in H^4(\Omega_0); \Delta \alpha = 0 \}$, let $\varepsilon > 0$ and let $B_\alpha = T_\alpha$. Then there exists an optimal control for the optimization problem \eqref{eq:state:optsystem2}.
\end{theorem}
\begin{proof}
 Lemma \ref{lemma:linf:regularity} yields the existence of a unique state $\Psi(\alpha) \in H^4(\Omega_0)$ for every control $\alpha \in \alphaspace$ and the state solves
 \begin{align}
 \label{eq:ex minimizer 1}
  \begin{aligned}
   \Delta e^{-2\alpha}\Delta \Psi(\alpha) &= 0 \qquad &&\mbox{in $\Omega_0$}
   \\
   \Psi(\alpha) &= g_0 \qquad &&\mbox{on $\Gamma_0$}
   \\
   \partial_\nv \Psi(\alpha) &= 0 \qquad &&\mbox{on $\Gamma_0$}.
  \end{aligned}
 \end{align}
 For $\Psi, \alpha \in H^4(\Omega_0)$ we define the cost functional 
 \begin{align}
  J(\Psi, \alpha) := \vectornorm{\sigma_d - \Delta \Psi}_{C^0(\Gamma_0^w)} 
   + \varepsilon \vectornorm{\alpha}_{H^4(\Omega_0)}^2.
 \end{align}
 
 Clearly, $\alphaspace$ is nonempty since $0 \in \alphaspace$. Furthermore, $J(\Psi, \alpha) \geq 0$ for all $(\Psi, \alpha) \in H^4(\Omega_0) \times \alphaspace$ and thus
 \begin{align}
  j := \inf_{\alpha \in \alphaspace} J(\Psi(\alpha), \alpha) \in \R^+_0
 \end{align}
 exists. We choose a minimizing sequence $(\Psi_n, \alpha_n) \in H^4(\Omega_0) \times \alphaspace$ such that
 \begin{align}
  \Psi_n = \Psi(\alpha_n) \qquad \mbox{and} \qquad J(\Psi_n, \alpha_n) \rightarrow j \quad \mbox{for} \quad n \rightarrow \infty.
 \end{align}
 Then, because of the regularization term and since $\varepsilon > 0$, there exists a constant $C > 0$ such that $\vectornorm{\alpha_n}_{H^4(\Omega_0)} \leq C$ and \cite{alt2006lineare} yields the existence of a weakly convergent subsequence of $\alpha_n$ which we again denote by $\alpha_n$, i.e.
 \begin{align}
  \alpha_n \rightharpoonup \bar{\alpha}, \quad n \rightarrow \infty
  \qquad \mbox{in $H^4(\Omega_0)$}
 \end{align}
 with $\bar{\alpha} \in H^4(\Omega_0)$. Because of the continuity of the Laplace operator $\alphaspace$ is closed and since it is also convex $\alphaspace$ is weakly closed (cf. \cite{alt2006lineare}) which yields $\bar{\alpha} \in \alphaspace$. Furthermore, the Lemma of Sobolev (cf. \cite{wloka1987partial}) yields a compact embedding which implies the strong convergence (cf. \cite{alt2006lineare})
 \begin{align}
 \label{eq:ex min conv alpha}
  \alpha_n \rightarrow \bar{\alpha}, \quad n \rightarrow \infty
  \qquad \mbox{in $C^2(\Omegab_0)$}.
 \end{align}
 
 Let $\gt \in H^4(\Omega_0)$ be an extension with $\gt|_{\Gamma_0} = g_0$ and $\partial_\nv \gt|_{\Gamma_0} = 0$ and define $\psi_n := \Psi_n - \gt \in H^4(\Omega_0) \cap H^2_0(\Omega_0)$. Then the standard existence and regularity theory (see \cite{wloka1987partial}) applied to the equation
 \begin{align}
  \Delta e^{-2\alpha_n}\Delta \psi_n = -\Delta (e^{-2\alpha_n}\Delta \gt) \qquad \mbox{in $\Omega_0$}
 \end{align}
 yields the estimate
 \begin{align}
  \vectornorm{\psi_n}_{H^4(\Omega_0)} \leq C\left(\vectornorm{e^{-2\alpha_n}}_{C^2(\Omegab_0)}\right) \vectornorm{\Delta (e^{-2\alpha_n}\Delta \gt)}_{L^2(\Omega_0)}.
 \end{align}
 Here $C(\vectornorm{e^{-2\alpha_n}}_{C^2(\Omegab_0)})$ is a constant depending on the coefficient $e^{-2\alpha_n}$ and this coefficient is bounded $0 < a_l \leq \vectornorm{e^{-2\alpha_n}}_{C^2(\Omegab_0)} \leq a_u$ since $\vectornorm{\alpha_n}_{H^4(\Omega_0)} \leq C$. Thus, $\vectornorm{\psi_n}_{H^4(\Omega_0)} \leq C_1$ is bounded with a constant $C_1 > 0$ independent of $n$. By \cite{alt2006lineare} there exists a weakly convergent subsequence of $\psi_n$
 \begin{align}
 \label{eq:ex min conv psi}
  \psi_n \rightharpoonup \bar{\psi}, \quad n \rightarrow \infty
  \qquad \mbox{in $H^4(\Omega_0) \cap H_0^2(\Omega_0)$}
 \end{align}
 with $\bar{\psi} \in H^4(\Omega_0) \cap H_0^2(\Omega_0)$. The Lemma of Sobolev yields a compact embedding which implies the strong convergence
 \begin{align}
 \label{eq:ex min strong conv psi}
  \Delta \psi_n \rightarrow \Delta \bar{\psi}, \quad n \rightarrow \infty
  \qquad \mbox{in $C^0(\Omegab_0)$}.
 \end{align}

 Since $\Psi_n$ and $\alpha_n$ solve \eqref{eq:ex minimizer 1} we know that $\psi_n$ and $\alpha_n$ solve the weak formulation
 \begin{align}
 \int_{\Omega_0} e^{-2\alpha_n} \Delta \psi_n \Delta \phi \dint x =-\int_{\Omega_0} \Delta (e^{-2\alpha_n} \Delta \gt) \phi \dint x
 \qquad \mbox{for all $\phi \in H^2_0(\Omega_0)$}.
\end{align}
Then \eqref{eq:ex min conv alpha} and \eqref{eq:ex min conv psi} are sufficient to pass to the limit and we conclude
\begin{align}
 \int_{\Omega_0} e^{-2\bar{\alpha}} \Delta \bar{\psi} \Delta \phi \dint x =-\int_{\Omega_0} \Delta (e^{-2\bar{\alpha}} \Delta \gt) \phi \dint x
 \qquad \mbox{for all $\phi \in H^2_0(\Omega_0)$}.
\end{align}
This shows that $\bar{\Psi} := \bar{\psi} + \gt$ and $\bar{\alpha}$ solve the state equation, i.e. $\bar{\Psi} = \Psi(\bar{\alpha})$.

Because of \eqref{eq:ex min strong conv psi} together with the continuity of the norm and since the $H^4(\Omega_0)$-norm is weakly lower semicontinuous (see \cite{alt2006lineare}) we conclude
\begin{align}
 J(\bar{\Psi}, \bar{\alpha}) \leq \liminf_{n \rightarrow \infty} J(\Psi_n, \alpha_n) = j
\end{align}
and thus $J(\bar{\Psi}, \bar{\alpha}) = j$.
\end{proof}

  \section{Set of Admissible Conformal Parameters}
  \label{sec:inflow constraints}
  In this section we want to redefine the set of admissible conformal parameters in a way that it is suitable for the application. An important property of the conformal parameter is that its influence is global. So if we change it in a small region the corresponding conformal domain changes everywhere. Or if we move the wall boundaries to change the wall shear stress, the inflow boundaries are moved as well. However, from the applications point of view one usually wants to keep the inflow boundaries fixed and only change the wall boundaries. Therefore, let us deal with the question of how certain features of the inflow boundaries, like length or curvature, can be preserved by applying constraints to the conformal parameter. Basically we do this by restricting the conformal parameter $\alpha \in \alphaspace$ by either homogeneous Dirichlet or Neumann boundary conditions. However, note that only apply these boundary conditions on the inflow boundaries and not on the wall boundaries. 
  
  In the following we define three different choices for the conformal parameter set $\alphaspace$ and discuss the effects: On the one hand we can use a Dirichlet condition
  \begin{align}
  \tag{$\mathcal{I}1$}
   \label{inflowbd_length}
   \alphaspace := \{\alpha \in H^4(\Omega_0);\; \Delta \alpha = 0; \;
   \alpha|_{\Gamma_0^{in}} = 0 \}
  \end{align}
  which assures that the boundary of the reference domain is mapped isometrically to the corresponding conformal boundary. Especially, this makes sure that the total length of the inflow boundaries is preserved. Unfortunately, this does not fix the curvature and thus a formerly straight inflow can bend. For this setup $T_\alpha = I_\alpha$ holds, so \eqref{eq:state:ConformalMapping} and \eqref{eq:state:IsometricMapping} are equivalent. The effects of \eqref{inflowbd_length} are illustrated in Figure \ref{fig:state:Example 1}.
  
  On the other hand we can use a Neumann condition
  \begin{align}
  \tag{$\mathcal{I}2$}
  \label{inflowbd_curvature}
  \alphaspace := \{\alpha \in H^4(\Omega_0);\; \Delta \alpha = 0; \;
  \partial_\nv \alpha|_{\Gamma_0^{in}} = 0 \}
  \end{align}
  which preserves the curvature. Thus, bending of the inflow boundaries is not possible, but now its length can change. Thus, Condition \eqref{state:eq:isometry assumption} does not hold and we can only use \eqref{eq:state:ConformalMapping}. See Figure \ref{fig:state:Example 2} for an example.
  
  Naturally we would like to be able to preserve curvature and length at the same time. This is possible by combining the Neumann constraint with an additional integral constraint
  \begin{align}
  \tag{$\mathcal{I}3$}
  \label{inflowbd_lengthcurvature}
  \begin{aligned}
  \alphaspace := \Bigg\{ &\alpha \in H^4(\Omega_0);\; \Delta \alpha = 0; \;
   \partial_\nv \alpha|_{\Gamma_0^{in}} = 0;\; \dots
   \\
   &\int_{\Gamma_0^k} e^{\alpha} \dint s = \int_{\Gamma_0^k} 1 \dint s
   \mbox{ for all $\Gamma_0^k \in \mathcal{C}(\Gamma_0^{in})$} \Bigg\}.
   \end{aligned}
  \end{align}
  The integral constraint makes sure that for every $\Omega_\alpha \in \admshape$ and $\Gamma_0^k \in \mathcal{C}(\Gamma_0^{in})$ the length of $\Gamma_0^k$ and $\Gamma_\alpha^k = T_\alpha(\Gamma_0^k)$ is equal. This can be seen by
  \begin{align}
    \int_{\Gamma_\alpha^k} 1 \dint s = 
    \int_{T_\alpha^{-1} \circ \Gamma_\alpha^k} 1 \abs{\partial_s T_\alpha} \dint s =
    \int_{\Gamma_0^k} e^{\alpha} \dint s.
  \end{align}
   In this case Condition \eqref{state:eq:isometry assumption} is fulfilled and we can either use \eqref{eq:state:ConformalMapping} or \eqref{eq:state:IsometricMapping}. Examples which illustrate the effects of the choices for $\alphaspace$ are compared in Figures \ref{fig:state:Example 1}, \ref{fig:state:Example 2} and \ref{fig:state:Example 3}.
  \begin{remark}
  Note that the conformal parameter $\alpha \in \alphaspace$ acts as the control of the optimization problem. Furthermore, since $\alpha$ must fulfill the Laplace equation together with either Dirichlet or Neumann boundary condition on the inflow $\Gamma_0^{in}$, we can also interprete $u :=  \alpha|_{\Gamma_0^w}$ as the control.
  \end{remark}
  
  \section{Discretization Scheme}
  \label{state:sec:discretization}
  
  In order to solve Problem \eqref{eq:state:OptSystemConformal} we derive a full discretization using finite elements. Therefore, let $\Omega_h$ be a triangulation approximating the reference domain $\Omega_0$ and let $\Gamma_h$ be the boundary of $\Omega_h$ which decomposes into the inflow $\Gamma_h^{in}$ and wall parts $\Gamma_h^w$. Let $V_h$ be a finite element space on $\Omega_h$, let $W_h$ be a finite element space on the boundary and let there exist a linear trace operator $G: V_h \rightarrow W_h$. Let $\Psi_h, \omega_h, \alpha_h \in V_h$ and let $g_0, \sigma_d \in W_h$.
  
  First we deal with \eqref{eq:state:ConformalMapping}, which leads to the simplification $\Psi = g$ on $\Gamma_0$. Then, there exist linear operators $A_{11}$ and $A_{12}$ and an inhomogeneous right hand side $b_1$ such that
   \begin{align}
   \label{eq:state:discrete cef}
    A_{11} \Psi_h + A_{12} (e^{2\alpha_h}\omega_h) &= b_1
   \end{align}
  approximates Equation \eqref{eq:state:OptSystemConformal c} together with the boundary conditions \eqref{eq:state:OptSystemConformal e} and \eqref{eq:state:OptSystemConformal f}. Here the operation $e^{2\alpha_h}\omega_h$ is to be interpreted pointwise. For details on computing the approximation we refer to any text book about finite elements (e.g. \cite{ern2004theory}). In our case we use the finite element software FreeFem++ (see \cite{pironneaufreefem++}) to compute the discretization.
  
   On the other hand when using \eqref{eq:state:IsometricMapping} we have to show how the Dirichlet condition
  \begin{align}
   \label{eq:isometric inflow condition1}
   \Psi &= g_0 \circ I_\alpha^{-1} \circ T_\alpha \qquad &&\mbox{on $\Gamma_h^{in}$}
   \\
   \label{eq:isometric inflow condition2}
   \Psi &= g_0 \qquad &&\mbox{on $\Gamma_h^w$}
  \end{align}
  can be transformed in such a way that it can be handled by the NLP-solver. Let $A_{11}$, $A_{12}$ and $b_1$ be chosen such that \eqref{eq:state:discrete cef} approximates Equation \eqref{eq:state:OptSystemConformal c} together with the boundary conditions \eqref{eq:isometric inflow condition2} and \eqref{eq:state:OptSystemConformal f}, i.e. it does not account for the inflow Dirichlet condition. To implement the inflow Dirichlet condition, let $\Gamma_h^k \in \mathcal{C}(\Gamma_h^{in})$ be an inflow part and let $\xi_0$ and $\xi_1$ be the first and the last point of $\Gamma_h^k$, respectively. We introduce the additional length variable $L_k \in W_h|_{\Gamma_h^k}$,  which we define by
  \begin{align}
  \label{eq:length constraint 1}
   L_k(\xi) = \int_{\xi_0}^{\xi} e^{\alpha_h} \dint s.
  \end{align}
  Let $s_k: [0, l_k] \rightarrow \Gamma_h^k$ be the isometric parameterization of the curve $\Gamma_h^k$, where $l_k$ is the length of that curve and $s_k(0) = \xi_0$ and $s_k(l_k) = \xi_1$ holds. Because of the necessary length constraint \eqref{state:eq:isometry assumption}
  \begin{align}
  \label{eq:length constraint 2}
   L_k(\xi_1) = \int_{\xi_0}^{\xi_1} 1 \dint s = l_k
  \end{align}
  holds and on $\Gamma_h^k$ Equation \eqref{eq:isometric inflow condition1} is equivalent to
  \begin{align}
   \Psi = g_0 \circ s_k \circ L_k.
  \end{align}
  Then, the function $g_k := g_0 \circ s_k$ can be precomputed and Equation \eqref{eq:length constraint 1} can be approximated by
  \begin{align}
   A_{k1} L_k + A_{k2} e^{\alpha_h} = 0
  \end{align}
  where $A_{k1}$ and $A_{k2}$ are linear operators.
  
  Moving on to the next Equation, there exists a linear operator $A_2$ such that
  \begin{align}
   A_2 \omega_h = 0
  \end{align}
  approximates Equation \eqref{eq:state:OptSystemConformal d}. The inequality constraints are interpreted pointwise by
  \begin{align}
   \sigma_d + G\omega &\leq \delta \qquad &&\mbox{on $\Gamma_h^w$}
   \\
   - \sigma_d - G\omega &\leq \delta \qquad &&\mbox{on $\Gamma_h^w$}.
  \end{align}
  
  We have to assure that $\alphaspace_h$ is a sufficient approximation of $\alphaspace$. When using \eqref{inflowbd_length} there exists a linear operator $A_{3}$ discretizing $\Delta \alpha = 0$ together with $\alpha = 0$ on $\Gamma_0^{in}$ and we can define
  \begin{align}
   \alphaspace_h := \{ \alpha_h \in V_h; \; A_{3} \alpha_h = 0 \}.
  \end{align}
  In the same way for \eqref{inflowbd_curvature} there exists a linear operator $A_{3}$ approximating $\Delta \alpha = 0$ together with $\partial_\nv \alpha = 0$ on $\Gamma_0^{in}$ and we define
  \begin{align}
   \alphaspace_h := \{ \alpha_h \in V_h; \; A_{3} \alpha_h = 0 \}.
  \end{align}
  For \eqref{inflowbd_lengthcurvature} let $A_{3}$ approximate $\Delta \alpha = 0$ together with $\partial_\nv \alpha = 0$ on $\Gamma_0^{in}$. Furthermore, the additional length constraint must hold which we can approximate using a linear operator $A_{4}$ and an inhomogeneous right hand side $b_4$. This yields
  \begin{align}
   \alphaspace_h := \{ \alpha_h \in V_h; \; A_{3} \alpha_h = 0; \; 
   A_{4} e^{\alpha_h} &= b_4 \}.
  \end{align}
  
 \begin{remark}
  While the the theoretical analysis has required a high regularity of the control space, we now relax this and only use a $H^1(\Omega_0)$-regularization instead of a $H^4(\Omega_0)$-regularization. As a compensation we apply additional box constraints on the control, i.e.
 \begin{align}
   \alpha_l \leq \alpha \leq \alpha_u \qquad \mbox{in $\Omega_0$}
  \end{align}
  with $\alpha_l, \alpha_u \in \R$. Moreover, $\alpha_l \leq 0 \leq \alpha_u$ should hold, because otherwise $\Omega_0 \notin \admshape$. The computed results show that this is sufficient for the numerics.
 \end{remark}

  Putting everything together, we end with an NLP of the following form
  \begin{align}
  \label{eq:state:OptSystemConformalDiscrete}
  \begin{aligned}
   \mbox{minimize}_{(\delta, \alpha_h, \Psi_h, \omega_h) \in M_h} &\omit\rlap{$\quad \delta
   +\varepsilon \vectornorm{\Grad \alpha_h}^2_{L^2(\Omega_0)}$}
   \\
   \mbox{with} \quad M_h =\R \times V_h &\times V_h \times V_h
   \\
   \mbox{subject to }A_{11} \Psi_h + A_{12} (e^{2\alpha_h}\omega_h) &= b_1
   \\
   \Psi|_{\Gamma_0^k} &= g_k(L_k)
   \qquad && \mbox{for $\Gamma_h^k \in \mathcal{C}(\Gamma_h^{in})$, if \eqref{eq:state:IsometricMapping}}
   \\
   A_{k1} L_k + A_{k2} e^{\alpha_h} &= 0
   \qquad && \mbox{for $\Gamma_h^k \in \mathcal{C}(\Gamma_h^{in})$, if \eqref{eq:state:IsometricMapping}}
   \\
   A_2 \omega_h &= 0
   \\
   A_3 \alpha_h &= 0
   \\
   A_{4} e^{\alpha_h} &= b_4
   \qquad && \mbox{if \eqref{inflowbd_lengthcurvature}}
   \\
   \alpha_l \leq \alpha_h &\leq \alpha_u \qquad &&\mbox{on $\Omegab_h$}
   \\
   \sigma_d + G\omega &\leq \delta \qquad &&\mbox{on $\Gamma_h^w$}
   \\
   -\sigma_d - G\omega &\leq \delta \qquad &&\mbox{on $\Gamma_h^w$}.
   \end{aligned}
  \end{align}
  This problem has the form of an NLP which can be solved using existing methods. In the following we use the interior point solver LOQO (see \cite{vanderbei1999interior}).
  
  \subsection{Reconstruction of the Domain}
  \label{state:sec:reconstruction}
  After an optimal solution of Problem \eqref{eq:state:OptSystemConformalDiscrete} has been computed, it remains to reconstruct the optimal domain $\Omega_{opt}$. Therefore, let $\alpha_{opt} \in \mathbb{P}_1$ be the projection of the optimal conformal parameter into the space $\mathbb{P}_1$, i.e. the space of Lagrangian finite elements of order one. Let $\theta \in \mathbb{P}_1^2$ represent the conformal map which is to be computed. Let $E$ be the set of edges of the triangulation of $\Omega_h$. For every edge $[i, j] \in E$ the corresponding vertex coordinates are denoted by $v_i, v_j \in \R^2$. By definition of the conformal parameter $e^{\alpha_{opt}}$ is the scaling factor of every infinite length element. Then, for every finite edge $[i, j] \in E$ the following should hold
  \begin{align}
   \abs{\theta(v_i)-\theta(v_j)} \approx e^{0.5(\alpha_i + \alpha_j)} \abs{v_i-v_j}
  \end{align}
  where $\alpha_i$ and $\alpha_j$ are the values of $\alpha_{opt}$ in $v_i$ and $v_j$, respectively. Considering this $\theta$ can be computed by minimizing the functional
  \begin{align}
  \label{state:eq:reconstruction}
   \mbox{minimize}_{\theta \in \mathbb{P}_1^2} \sum_{[i, j] \in E}
   \left( 
   \abs{\theta(v_i)-\theta(v_j)}^2 - e^{(\alpha_i + \alpha_j)} \abs{v_i-v_j}^2
   \right)^2.
  \end{align}
  Computing the solution $\theta$ is easy and can be done using existing methods. The optimal domain is then given by $\Omega_{opt} := \theta(\Omega_h)$.

  This gives rise to Algorithm \ref{algo:linf state constraint} for supremum norm shape optimization problems with state constraints.
  \begin{algorithm}[htb]
\caption{$C^0$ Shape Optimization with State Constraints (2D)}
\label{algo:linf state constraint}
 \begin{algorithmic}[1]
 \STATE Let the initial domain $\Omega_0$ and target wall shear stress $\sigma_d$ be given.
 \STATE Solve \eqref{eq:state:OptSystemConformalDiscrete} using an NLP solver, which yields the optimal conformal parameter $\alpha_{opt}$.
 \STATE Use $\alpha_{opt}$ to reconstruct the optimal domain $\Omega_{opt} := \theta(\Omega_h)$ by solving \eqref{state:eq:reconstruction}.
\end{algorithmic}
\end{algorithm}

\section{Numerical Results}
\label{sec:numerical results}

We want test the proposed method using two different geometric scenarios. The first is a simple rectangular geometry which we use to demonstrate the effect of the different choices of the conformal parameter set $\alphaspace$ (cf. Section \ref{sec:inflow constraints}). The second geometry is a distributor with a small inflow tube at the top and a broad outflow at the bottom, which is inspired by the industrial application of a polymer flow distributor. In the end we provide statistics for all examples on different mesh sizes.

Note that both geometries are smooth except for a finite set of corner points. Hence, we have dropped the assumption that $\Omega_0$ is of class $C^{4,1}$. The numerics do still work but we must consider the following fact: Conformal maps are angle preserving, therefore, the angle of each corner is fixed and cannot be altered. Thus, the set of admissible shapes does only contain geometries whose corner points have the same angles than the reference geometry.

Let $V_h$ be the space of second order Lagrangian elements on $\Omega_0$ and let $W_h$ be the space of first order Lagrangian elements on the boundary. We use FreeFem++ \cite{pironneaufreefem++} to generate the mesh as well as to assemble the problem and compute the finite element matrices necessary for the NLP \eqref{eq:state:OptSystemConformalDiscrete}. Solving the NLP is the crucial step which can be effectively done using the interior point solver LOQO \cite{vanderbei1999interior} and its interface to the algebraic modeling language AMPL \cite{fourer2002ampl}.

\subsection{Rectangular Geometry}
\begin{figure}[p]
\begin{leftfullpage}
  \centering
  \includegraphics{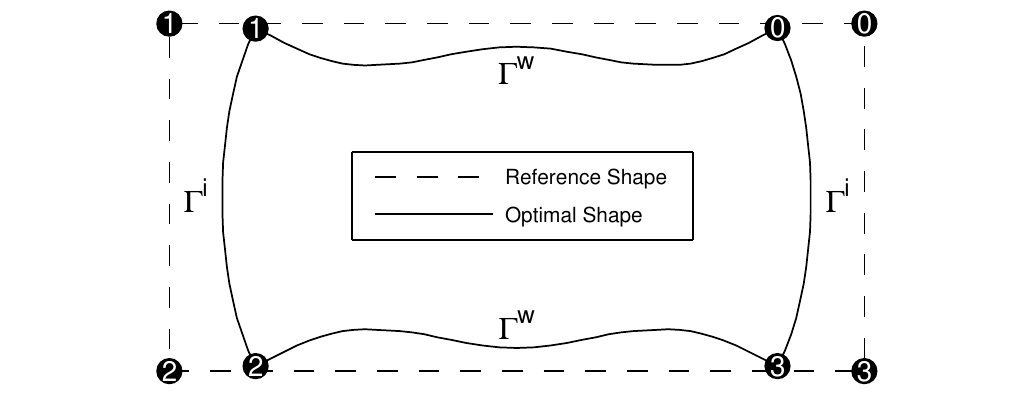}
  \\
  \includegraphics{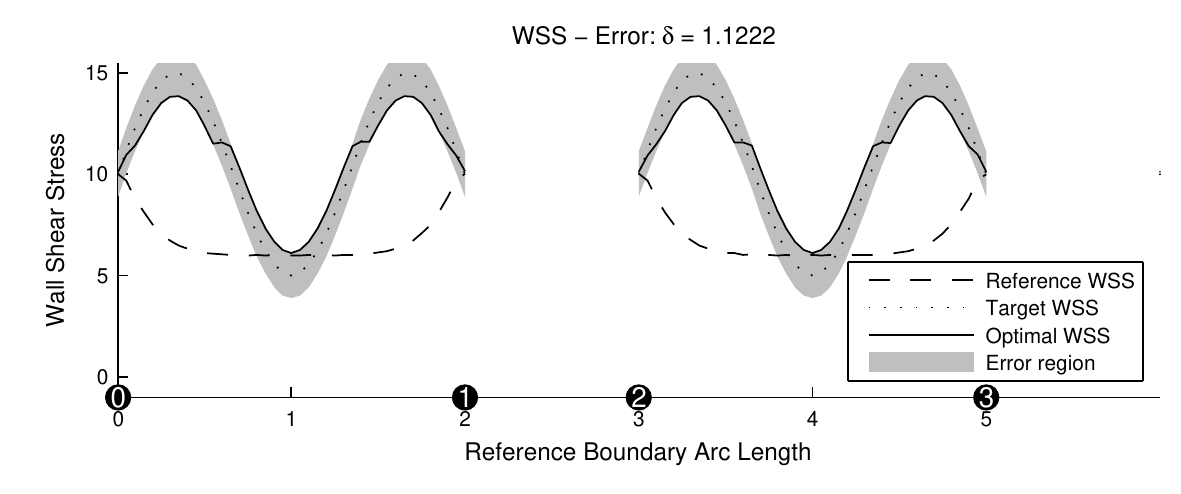}
  \caption{Example 1: Optimal solution using \eqref{inflowbd_length} and \eqref{eq:state:ConformalMapping}/\eqref{eq:state:IsometricMapping} which coincide.}
  \label{fig:state:Example 1}

  \includegraphics{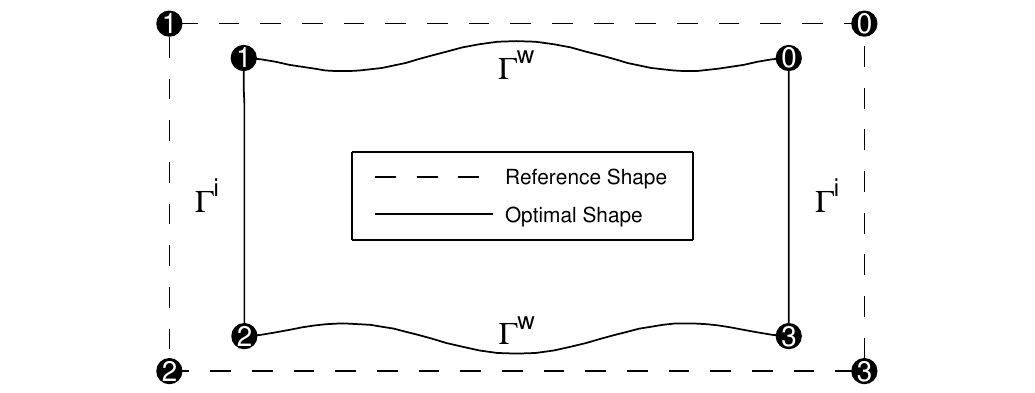}
  \\
  \includegraphics{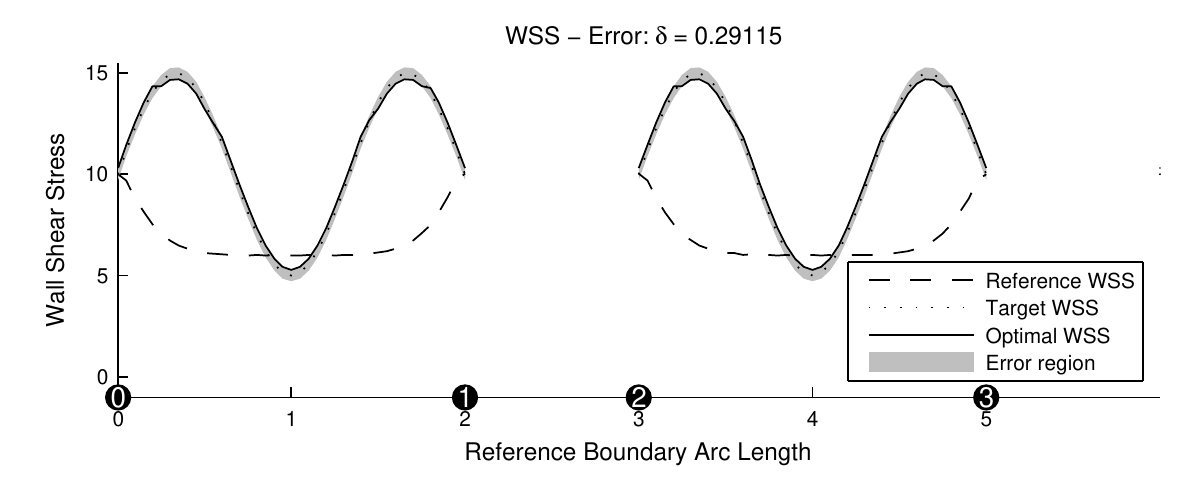}
  \caption{Example 2: Optimal solution using \eqref{inflowbd_curvature} and \eqref{eq:state:ConformalMapping}.}
  \label{fig:state:Example 2}
  \end{leftfullpage}
  \end{figure}
  
  \begin{figure}[p]
  \begin{fullpage}
  \centering
  \includegraphics{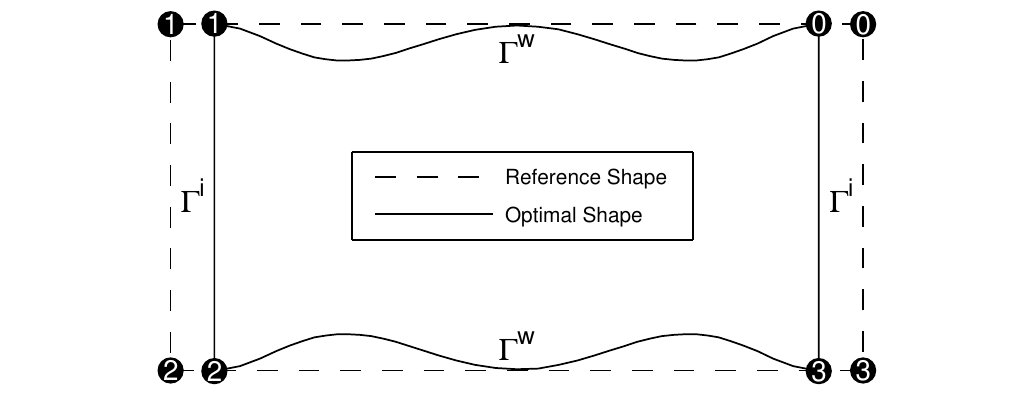}
  \\
  \includegraphics{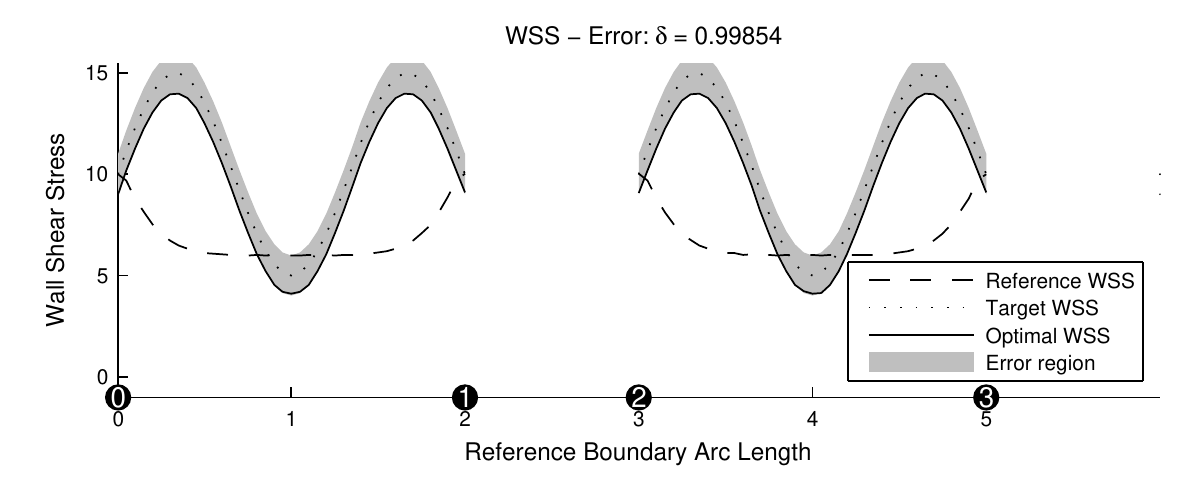}
  \caption{Example 3: Optimal solution using \eqref{inflowbd_lengthcurvature} and \eqref{eq:state:IsometricMapping}.}
  \label{fig:state:Example 3}
  
  \includegraphics{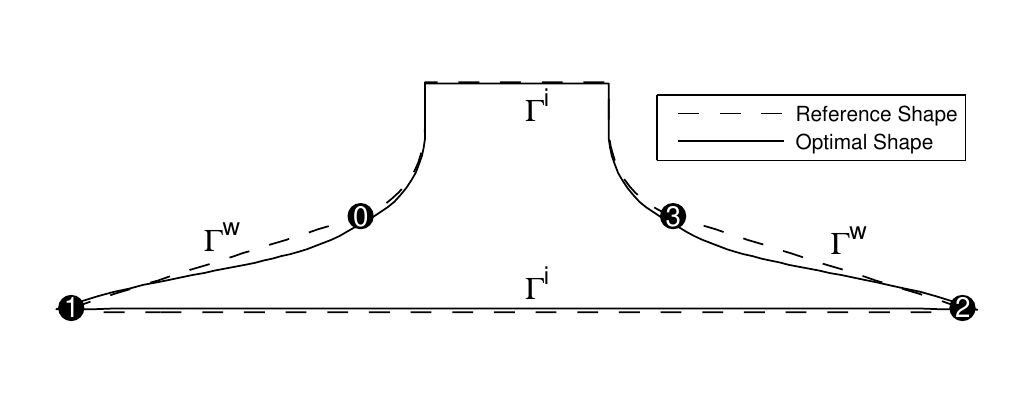}
  \\
  \includegraphics{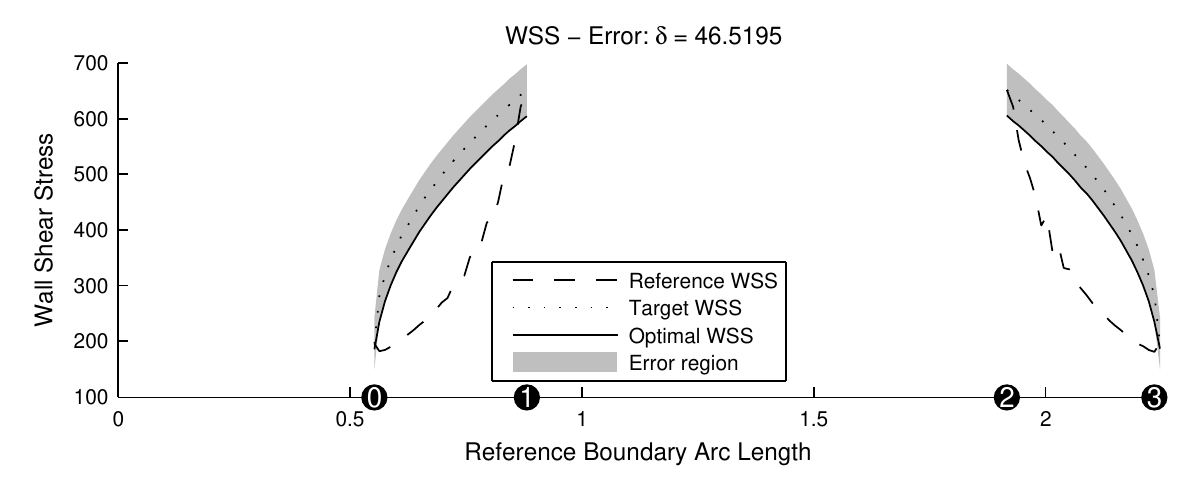}
  \caption{Example 4: Optimal solution using \eqref{inflowbd_lengthcurvature} and \eqref{eq:state:IsometricMapping}.}
  \label{fig:state:Example 4}
  \end{fullpage}
  \end{figure}
  In the first three test cases we compute solutions of Problem \eqref{eq:state:OptSystemConformalDiscrete} on a simple rectangular domain given by
\begin{align}
 \label{eq:state:Tube Domain}
 \Omega_0 = \{(x, y) \in \R^2; -1 < x < 1; -0.5 < y < 0.5 \}.
\end{align}
 Let the left and right sides of $\Omega_0$ be the inflow boundaries $\Gamma_0^{in}$ and let the top and bottom sides be the wall boundaries $\Gamma_0^w$. Let the inflow condition $u_0 \in H^{\frac{1}{2}}(\Gamma_0)$ be defined by
\begin{align}
 u_0(x, y) =  20 \; \sign (x)\left( 0.5^4 -y^4 \right).
\end{align}
This condition is chosen in such a way that $u_0$ vanishes on the wall boundaries and that the total amount flowing in and out of the domain is normalized to one. Integration over the boundary yields
\begin{align}
 \label{eq:state:Tube Psi_0}
 g_0(x, y) = 20 \cdot 0.5^4 \; y - 4 y^5.
\end{align}

\begin{remark}
\label{rm:corner points}
Note that under sufficient regularity assumptions the value of $\Delta \Psi$ in a corner point is already determined by the boundary condition $g_0$. This is due to the fact that there are two independent boundary directions for that corner point and thus the value of $\sigma = \Delta \Psi$ is determined by the stream function boundary condition. Due to this property it is reasonable to use target wall shear stresses $\sigma_d$ which agree with the intrinsic condition in the corner points.
\end{remark}

On account of this remark we define the following target wall shear stress
\begin{align}
\label{eq:state:Tube sigma_d}
 \sigma_d(x, y) = \sign(y)(-5 \cos(1.5 \pi x) + 10)
\end{align}
for $(x, y) \in \Gamma_0$. See Figures \ref{fig:state:Example 1}-\ref{fig:state:Example 3} for an illustration. Furthermore, we choose the regularization parameter $\varepsilon = 0.01$ and the control constraints $\alpha_l = -0.45$ and $\alpha_u = 0.45$. The control constraints are chosen in such a way that they are active in all test cases. See Table \ref{state:table:performanceinf} for results with inactive control constraints which show that the target wall shear stress is actually reachable if the control is unconstrained.

Using this setup we compute the optimal solutions for the following three test cases, where our goal is to compare the effects of the different choices of $\alphaspace$.
\begin{itemize}
 \item Example 1: Using \eqref{inflowbd_length} and \eqref{eq:state:ConformalMapping}/\eqref{eq:state:IsometricMapping} which are equivalent.
 \item Example 2: Using \eqref{inflowbd_curvature} and \eqref{eq:state:ConformalMapping}.
 \item Example 3: Using \eqref{inflowbd_lengthcurvature} and \eqref{eq:state:IsometricMapping}.
\end{itemize}

The results are shown in Figures \ref{fig:state:Example 1}, \ref{fig:state:Example 2} and \ref{fig:state:Example 3}. We have used a mesh with $953$ vertices. The first plot compares the reference shape $\Omega_0$ with the optimal shape $\Omega_{opt}$. The second one is a plot over the arc length of the reference boundary, which shows the absolute values of the reference wall shear stress on $\Omega_0$, the target wall shear stress $\sigma_d$ and the optimal wall shear stress on $\Omega_{opt}$ together with the error region, i.e. the region between $\sigma_d - \delta$ and $\sigma_d + \delta$. The black markers help to draw a connection between geometry and arc length plot. In all three cases we have succeeded to drive the wall shear stress close to the target wall shear stress.

From Figure \ref{fig:state:Example 1} we can see that \eqref{inflowbd_length} has preserved the length of both inflow boundaries, but their curvature has changed. On the other hand \eqref{inflowbd_curvature} has kept the inflow straight but it was shortened, as we can see in Figure \ref{fig:state:Example 2}. Only \eqref{inflowbd_lengthcurvature} has preserved both characteristic properties, as shown by Figure \ref{fig:state:Example 3}.

These examples have illustrated how constraints on the conformal parameter can preserve characteristic properties of the inflow boundaries. With the application in view it makes sense to use \eqref{inflowbd_lengthcurvature} and \eqref{eq:state:IsometricMapping}, because these preserve length and curvature of the inflow and leave the inflow condition unaltered. The only thing which can still change is the relative position of two inflow parts to each other as we can see from Figure \ref{fig:state:Example 3}.

\begin{figure}[htb]
\centering
\subfigure[Example 1]{\includegraphics{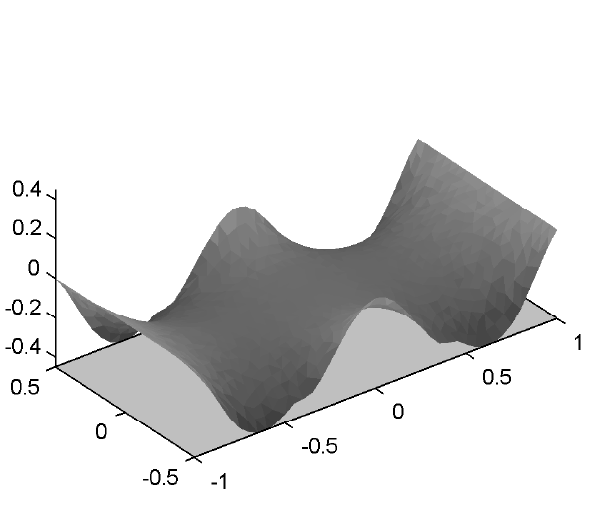}}
\subfigure[Example 2]{\includegraphics{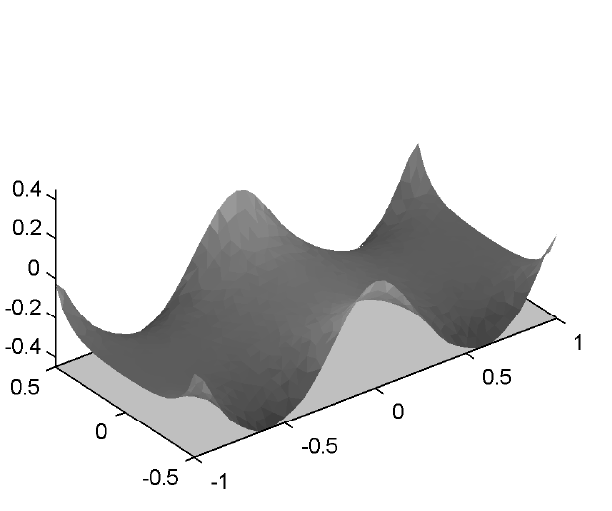}}
\\
\subfigure[Example 3]{\includegraphics{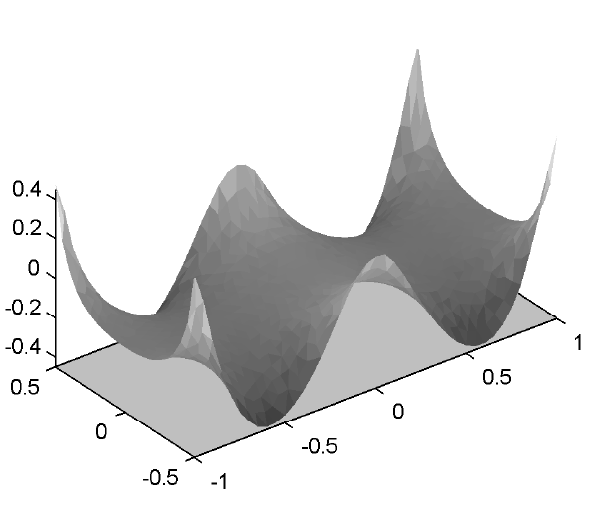}}
\subfigure[Example 4]{\includegraphics{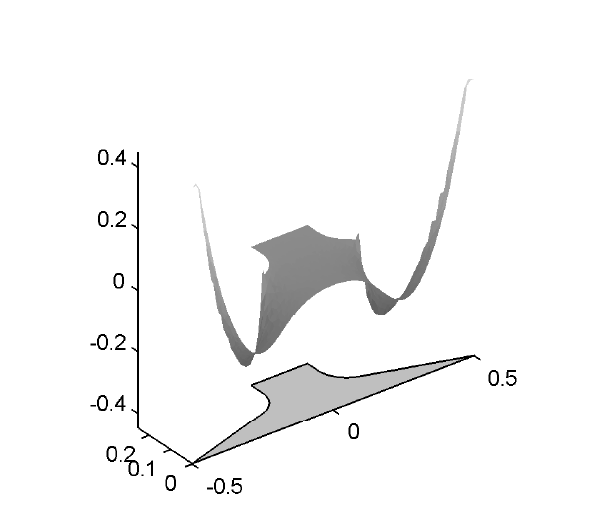}}
  \caption{Optimal conformal parameters $\alpha_{opt}$ corresponding to the test cases. Note that the control constraints $-0.45 \leq \alpha \leq 0.45$ are active in all four cases.}
  \label{state:fig:conformal parameter}
\end{figure}

  \subsection{Distributor Geometry}
  The previous examples have shown the general functionality of the approach. Next we discuss a distributor geometry depicted in Figure \ref{fig:state:Example 4}. We have a small inflow tube at the top from where the geometry widens into a broad outflow at the bottom.
  
  Let $\Omega_0$ be the reference domain as shown in Figure \ref{fig:state:Example 4}. The inflow condition $u_0$ is defined in the following way. Let $s_t: [0, 0.4] \rightarrow \Gamma_0^{i_t}$ and $s_b: [0, 1] \rightarrow \Gamma_0^{i_b}$ be isometric parameterizations of the top and bottom boundary, respectively. Then, we define
  \begin{align}
   \begin{aligned}
   u_0(s_t(t)) &= \frac{375}{4} ((t-0.2)^2 - 0.2^2) \qquad && 
   \mbox{$t \in [0, 0.4]$, i.e. on $\Gamma_0^{i_t}$}
   \\
   u_0(s_b(t)) &= -\frac{5632}{5} ((t-0.5)^{10} - 0.5^{10}) \qquad && 
   \mbox{$t \in [0, 1]$, i.e. on $\Gamma_0^{i_b}$}
   \\
   u_0 & = 0 \qquad && 
   \mbox{on $\Gamma_0^w$}
   \end{aligned}
  \end{align}
  As in the previous example $u_0$ is scaled in such a way that the total amount flowing into the domain and the total amount flowing out is normalized to one. The boundary condition for the stream function $g_0$ can be obtained by integration of $u_0$ over the boundary. The next task is to define the target wall shear stress $\sigma_d$. Therefore, let $s_l: [0, l_l] \rightarrow \Gamma_0^{w_l}$ and $s_r: [0, l_r] \rightarrow \Gamma_0^{w_r}$ be isometric parameterizations of the left and right wall boundary, respectively and let $\sigma_0$ be the wall shear stress on the reference domain. Due to Remark \ref{rm:corner points} we define $\sigma_d$ such that it agrees with the reference wall shear stress $\sigma_0 = \sigma(\Omega_0)$ in the corner points of the domain. We define
  \begin{align}
   \begin{aligned}
    \sigma_d(s_l(t)) &= \left( 1- \sqrt{\frac{t}{l_l}} \right) \sigma_0(s_l(0)) + \sqrt{\frac{t}{l_l}} \sigma_0(s_l(l_l)) \quad && 
   \mbox{$t \in [0, l_l]$}
    \\
    \sigma_d(s_r(l_r-t)) &= \left( 1- \sqrt{\frac{t}{l_r}} \right) \sigma_0(s_r(l_r)) + \sqrt{\frac{t}{l_r}} \sigma_0(s_r(0)) \quad && 
   \mbox{$t \in [0, l_r]$}.
   \end{aligned}
  \end{align}
  To clarify this we refer to the plot of the target wall shear stress $\sigma_d$ in Figure \ref{fig:state:Example 4}.
  
  In order to get realistic results, we want that both curvature and length of the lower inflow boundary are preserved and that the inflow condition is mapped isometrically, i.e. we use \eqref{inflowbd_lengthcurvature} and \eqref{eq:state:IsometricMapping}. On the upper inflow boundary we use \eqref{inflowbd_length}, but as we can see from the result bending is not a problem. Furthermore, we use the regularization parameter $\varepsilon = 0.1$ and control constraints $\alpha_l = -0.45$ and $\alpha_u = 0.45$.
  
  We have used a mesh with 967 vertices and results are shown in Figure \ref{fig:state:Example 4}. Again we have succeeded to reach a wall shear stress close to the target stress and the characteristic properties of the geometry have been preserved.
  
  \begin{table}[htb]
 \begin{center}
  \begin{tabular}{l|rrrrrr}
  & \#Vertices & \#Variables & \#Iterations & Time [s] & $\delta$
  \\
  \hline
  Example 1
  &     262 &     2840 &     33 &      16 &    0.98 \\ Rectangle
  &  \it953 & \it10832 &  \it41 &  \it211 & \it1.12 \\
  &    2097 &    24260 &     52 &    1255 &    1.17 \\
  &    3731 &    43568 &     86 &    7299 &    1.19 \\
  \hline
  Example 2
  &     262 &     2878 &     32 &      16 &    0.15 \\ Rectangle
  &  \it953 & \it10910 &  \it44 &  \it197 & \it0.29 \\
  &    2097 &    24378 &     56 &    1542 &    0.32 \\
  &    3731 &    43726 &     87 &    5707 &    0.32 \\
  \hline
  Example 3
  &     262 &     2954 &     39 &      33 &     0.91 \\ Rectangle
  &  \it953 & \it11066 &  \it55 &  \it412 &  \it1.00 \\
  &    2097 &    24614 &     41 &    1479 &     1.12 \\
  &    3731 &    44042 &     45 &    5606 &     1.14 \\
  \hline
  Example 4
  &     253 &     2705 &     58 &      19 &     46.30 \\ Distributor
  &  \it967 & \it10941 &  \it83 &  \it335 &  \it46.52 \\
  &    2076 &    23927 &     73 &    1458 &     56.86 \\
  &    3746 &    43613 &    103 &    5035 &     88.70 \\
  \end{tabular}
  \end{center}
  \caption{Performance overview for the previous examples on different mesh sizes with active control constraints $\alpha_l = -0.45$ and $\alpha_u = 0.45$. The mesh cases which have been used for Figures \ref{fig:state:Example 1}-\ref{fig:state:Example 4} are marked in \it italics.}
  \label{state:table:performance45}
\end{table}

  \subsection{Numerical Reconstruction of the Domain}
  The optimal conformal parameters $\alpha_{opt}$ for the previous test cases are plotted in Figure \ref{state:fig:conformal parameter}. They are used to reconstruct the optimal shapes which are shown in Figures \ref{fig:state:Example 1}-\ref{fig:state:Example 4}. As mentioned in Section \ref{state:sec:reconstruction} this is relatively easy by solving the minimization problem \eqref{state:eq:reconstruction}. For this task we have again used LOQO which has no struggle computing the solution and the computation time is small compared to the time used for solving the main NLP.

\subsection{Solver Performance with Control Constraints}
  For all discussed examples Table \ref{state:table:performance45} shows an over\-view of the solver performance for different mesh sizes. The first column contains the total number of mesh vertices in the discretization. The mesh sizes for the rectangular and distributor geometry have been chosen in a way that the results are comparable. The second and third column show the number of variables in the NLP problem and the number of iterations needed by the LOQO solver. The fourth column containts the computation time  in seconds for solving the NLP, not including the time used for pre- and postprocessing. Of course the computation time depends on the PC infrastructure used, but here it is primarily meant for comparison. The last column gives the value of $\delta$, which is the absolute error between wall shear stress and target wall shear stress.

  From Table \ref{state:table:performance45} we can see that for all examples the value of $\delta$ increases when the size of the mesh is reduced. This may seem unexpected but the following two explanations may hold: In the discrete case the state constraints which determine $\delta$ must only hold pointwise in every boundary vertex. Thus, a bigger discretization error can lead to a smaller value of $\delta$. And the performance of the NLP solver may decrease with increasing complexity of the problem such that the computed solution lies further away from the optimum.
  
  \subsection{Solver Performance with Inactive Control Constraints}
  Table \ref{state:table:performanceinf} shows results for the same test cases as before, but computed with control constraints chosen in a way that they are inactive ($\alpha_l = -1$ and $\alpha_u = 1$) in the optimal solution of the respective examples. It turns out that the value of $\delta$ is almost zero for all cases, which means that the target wall shear stress is reachable. For the results in Table \ref{state:table:performance45} we have used tight control constraints to essentially restrict the control space and thus the target wall shear stress has become unattainable. So the reason to use tighter control constraints was to restrict the set of conformal parameters in such a way that the target wall shear stress is not reachable, in order to get a good test case for our supremum norm shape optimization approach.
  
  \begin{table}[htb]
 \begin{center}
  \begin{tabular}{l|rrrrrr}
  & \#Vertices & \#Variables & \#Iterations & Time [s] & $\delta$
  \\
  \hline
  Example 1
  &     262 &     2840 &     21 &      10 &    1.0e-09 \\ Rectangle
  &     953 &    10832 &     27 &     137 &    5.6e-11 \\
  &    2097 &    24260 &     29 &     679 &    3.6e-11 \\
  &    3731 &    43568 &     31 &    2506 &    5.8e-11 \\
  \hline
  Example 2
  &     262 &     2878 &     22 &      11 &    2.9e-10 \\ Rectangle
  &     953 &    10910 &     27 &     119 &    3.9e-10 \\
  &    2097 &    24378 &     30 &     807 &    7.2e-11 \\
  &    3731 &    43726 &     33 &    2499 &    8.1e-11 \\
  \hline
  Example 3
  &     262 &     2954 &     31 &      22 &     3.8e-10 \\ Rectangle
  &     953 &    11066 &     27 &     226 &     8.5e-11 \\
  &    2097 &    24614 &     30 &    1254 &     4.9e-10 \\
  &    3731 &    44042 &     32 &    4536 &     3.3e-10 \\
  \hline
  Example 4
  &     253 &     2705 &     56 &      18 &     5.4e-09 \\ Distributor
  &     967 &    10941 &     75 &     303 &     6.2e-10 \\
  &    2076 &    23927 &     90 &    1800 &     1.3e-15 \\
  &    3746 &    43613 &    116 &    5668 &     1.6e-12 \\
  \end{tabular}
  \end{center}
  \caption{Performance overview for the previous examples on different mesh sizes with inactive control constraints $\alpha_l = -1$ and $\alpha_u = 1$.}
  \label{state:table:performanceinf}
\end{table}

Note that we could have obtained the same results by disabling the control constraints, i.e. $\alpha_l = -\infty$ and $\alpha_u = \infty$, but as one may have expected the NLP solver shows a better performance if control constraints are available, even if they are not active. Thus it is in general a good idea to apply control constraints, even when they are not necessary for the application.
  
  \section{Conclusion}
  We have presented a general numerical approach to solve shape optimization problems with state constraints on two-dimensional geometries, by illustrating how the shape-dependent problem can be transformed into a nonlinear problem on a fixed reference domain using conformal pull-back. And we have demonstated that the structure of the nonlinear problem is such that it can be solved by modern NLP solvers like LOQO. Furthermore, we have suggested constraints on the conformal parameter that preserve the shape of the inflow boundaries and therewith the characteristics of the geometry.
  
  It is relatively easy to transfer the approach to a wide class of problems on two-dimensional domains with many different constraints. However, when moving on to higher dimensional geometries problems arise: It is still possible to use conformal maps, but there is no Riemann Mapping Theorem and the class of reachable geometries would be negligibly small. Of course one could use more general mappings, but this would increase the complexity of the problem, which would already be quite high in the three-dimensional setup.
  
  \section*{Acknowledgments}
  This work was supported by the German Federal Ministry of Education and Research (BMBF) grant no. 03MS606F.


\end{document}